\documentclass{article}
\usepackage{graphicx} % Required for inserting images
\usepackage[a4paper,width=150mm,left=2cm,right=2cm,top=3cm,bottom=3cm,bindingoffset=1cm]{geometry}
\usepackage[usenames,dvipsnames]{color}
\usepackage{fancyhdr}
\usepackage{mathpazo}
\usepackage{amsmath,amssymb}
\usepackage{enumitem}
\usepackage{graphics,color}
\usepackage{epsfig}
\usepackage[english]{babel}
\usepackage[autostyle, english = american]{csquotes}
\MakeOuterQuote{"}
\usepackage{physics}
\usepackage{graphicx}
\usepackage{amsfonts,amsmath,amssymb,amscd}
\usepackage{xcolor}
\usepackage{amsfonts,amsbsy,amsmath,amssymb,amsthm}
\usepackage{color,xcolor}
\usepackage{graphics,graphicx}
\usepackage{float,subcaption,longtable,multicol,multirow}
\usepackage{epsfig,epstopdf}
\usepackage[utf8]{inputenc}
\usepackage[T1]{fontenc}
\usepackage{mathtools}
\usepackage{authblk}
\usepackage{lscape}
\usepackage{physics}
\usepackage{hyperref,cite}
\usepackage{tikz}
\usepackage{setspace}
\usepackage[toc,page]{appendix}
\usepackage{enumitem}
\setlist[enumerate,1]{label=(\arabic*)}
\newtheorem{definition}{Definition}[section]
\newtheorem{theorem}{Theorem}[section]
\newtheorem{lemma}{Lemma}[section]

\newtheorem{corollary}{Corollary}[section]

\title{A supercharacter analog of vanishing-off subgroups and generalized Camina pair}
\author{Fahim Sayed}
\affil{Department of Mathematics, B.N. College (Autonomous), Dhubri, Assam, India}

\begin{document}
\date{}
\maketitle

\begin{abstract}
    Vanishing-off subgroups, generalized Camina pair and other related subgroups have played a significant role in the study of group structure. The primary goal of this paper is to study their analogs in the setting of supercharacter theory. We establish several properties of these subgroups which includes connections with supercharacter theory products. 
\end{abstract}
\section{Introduction}
	Character theory has contributed immensely to the systematic study of finite groups.  However, in certain cases,  such as the group of unipotent upper triangular matrices, $UT_n(\mathbb{F}_q)$ over  the finite field $\mathbb{F}_q$, it is known to be extremely challenging to find all the irreducible characters. This problem was tackled by Andr\'e \cite{Andre2001TheGroup} who found that certain  character sums exhibit most of the core properties of irreducible characters. Formalizing the work of Andr\'e \cite{Andre2001TheGroup}, Diaconis and Isaacs \cite{Diaconis2008SupercharactersGroups} introduced the supercharacter theory. This theory provides a unified way to define some complex-valued functions and partitions on a group resembling the irreducible characters and conjugacy classes of the group, respectively. For a supercharacter theory $S$ of a group $G$, the functions called the supercharacters or $S$-characters are constant on each part, called superclasses or $S$-classes of a partition of $G$. We denote the set of $S$-classes by Cl$(S)$ and that of  $S$-characters by  $\text{Irr}(S)$. The supercharacter theory, besides being an interesting study in its own right, also has numerous applications. For example, see \cite{Brumbaugh2014SupercharactersPrinciple, Garcia2019SupercharactersSums, Chavez2018SupercharactersSums, Sayed2022EighthCurves,  Sayed2023MomentsCurves, garcia2025moments}.
	\par The introduction of supercharacter theory naturally raises the problem of translating the classical group theoretic concepts to this new setting. Recently, several studies have been devoted in this direction. It began with the $S$-normal subgroup defined by Hendrickson \cite{Hendrickson2009ConstructionGroups}  analogously to a normal subgroup. A subgroup $N$ of $G$ is said to be $S$-normal, denoted as $N\triangle_S G$, if $N$ is the union of some $S$-classes.  In \cite{Marberg2011ANormality}, Marberg classified the set of all $S$-normal subgroups of the group $UT_n(\mathbb{F}_q)$, and provided a formula for the number of such subgroups for a prime $q$.   Later, Burkett  considered the supercharacter analog of nilpotent groups, called $S$-nilpotent groups in \cite{Burkett2020AnTheory} where  a significant amount of work has been done.  Extending the work of Burkett, the author established new properties and characterizations of $S$-nilpotent groups in \cite{sayed2024supercharacter}. Continuing along this line,  we focus here on the supercharacter analog of the vanishing-off subgroups, generalized Camina pair and related concepts. 
    \par In \cite[p.~200]{isaacs1994character}, Isaacs defines the \emph{vanishing-off subgroup} $V(\chi)$ of a character $\chi$, not necessarily irreducible, as the subgroup generated by those elements of $G$ for which $\chi$ is non-zero. In other words, $V(\chi)$ is the smallest subgroup $V $ of $G$ such that $\chi$ vanishes on $G \setminus V$. Later, Lewis studied these subgroups extensively in \cite{lewis2009vanishing}. Lewis also defines for a non-abelian group $G$, the \emph{vanishing-off subgroup} $V(G)$ of $G$ as the subgroup  generated by the elements $g \in G$ such that there exists some nonlinear character $\chi \in \text{Irr}(G)$ for which $\chi(g) \neq 0$.
    Here, $\text{Irr}(G)$ denotes the set of all irreducible characters of $G$. In classical group theory, the vanishing-off subgroups have been useful in determining links between the structure of $G$ and arithmetic properties of the degrees of its irreducible characters. 
    Also,  Mlaiki, \cite{mlaiki2014camina} and Lewis,\cite{lewis2009vanishing} used the vanishing-off subgroups to study Camina pairs and Camina triples.  A pair $(G,N)$ is called a \emph{Camina pair} if $N$ is a normal subgroup of $G$ for which every conjugacy class of $G$ lying outside of $N$ is a union of full $N$-cosets. Equivalently, every member 
of $\text{Irr}(G \mid N)$ vanishes off of $N$, where  $\text{Irr}(G \mid N)$ 
denotes the set of all irreducible characters of $G$ not containing $N$ 
in their kernel. For normal subgroups, $M$ and $N$, this has been naturally generalized to $(G,N,M)$, which is called a \emph{Camina triple} if for all $g \in G \setminus N$ and $m \in M$, the element $gm$ is conjugate to $g$. Another reason why the supercharacter analog of the vanishing-off subgroups deserves more attention is because in \cite{Burkett2020Vanishing-offProducts},  Burkett and Lewis had  noticed that for the usual supercharacter theory $m(G)$, there is a link between $\ast$-product, $\Delta$-product of supercharacter theory with Camina pair and Camina triples respectively. This was used to establish a strong connection between the analogs of vanishing-off subgroups and products of supercharacter theories. Thus, building upon the work of  Burkett and Lewis in \cite{Burkett2020Vanishing-offProducts}, we here explore some more properties of these and other related subgroups.
 To this end, we recall  from  \cite{Burkett2020Vanishing-offProducts} that for an $S$-normal subgroup $N$  of $G$, $V(S\mid N)$ is defined as 
			\begin{align*}
				V(S\mid N)=\langle g\in G: \text{there exists $\chi\in Irr(S \mid N)$ such that $\chi(g)\neq 0$}\rangle,
			\end{align*}
          where $\text{Irr}(S \mid N)$ is the set of $S$-characters not containing $N$ in their kernels. Also, let  $\text{Irr}(S/N)$ be the set of $S$-characters with $N$ in their kernels. For $H\leq G$, Burkett defined $[H,S]$ to be the subgroup generated by elements of the form  $g^{-1}k$ with $g \in H$ and  $k\in\text{Cl}_S(g)$, where $\text{Cl}_S(g)$ denotes the $S$-class containing $g$. By this notation, ${V}(S)$ is defined as  ${V}(S\mid [G,S])$.
          \par Motivated by the work of Burkett and Lewis in \cite{Burkett2020Vanishing-offProducts}, we define an element $g\in G$  to be an \emph{$S$-Camina element} if $\chi(g)=0$, for all $\chi\in$ Irr$\left(S\mid[G,S]\right)$. Furthermore, if every element in $G\setminus N$ is an $S$-Camina element, for an $S$-normal subgroup $N$ of $G$, we call $(G,N)$ to be a \emph{generalised $S$-Camina pair}, to be abbreviated as $S$-GCP. We begin our investigation with the following characterization of $S$-GCP. 
	\begin{theorem}\label{corgcp}
	    Let $N$ be an $S$-normal subgroup of $G$. Then, the following are equivalent
		\begin{enumerate}
			\item $(G,N)$ is an $S$-GCP 
            \item $S$ is a $\triangle$-product over $N$ and $[G,S]$, when $N\leq [G,S]$.
			\item $\abs{Cl_S(g)}=\abs{[G,S]}$ for all $g\in G\setminus N$.
			\item  For any $g\in G\setminus N$ and $z\in [G,S]$, there exists a $y\in Cl_S(g)$ such that $g^{-1}y=z$.
			\item Every character in Irr$\left(S\mid[G,S]\right)$ vanishes on $G\setminus N$.
		\end{enumerate}
	\end{theorem}
    Since, for $S$-normal subgroups $M$ and $N$ of $G$, the condition that  every $S$-class lying outside of $N$ is a union of $M$-cosets appears to be the correct analog of Camina triple, we call $(G,N,M)$ to be an \emph{$S$-Camina triple}. In the event that $N=M$, $(G,N,N)$ is shortened to $(G,N)$ and is called an \emph{$S$-Camina pair}. It is to be noted that $S$-Camina pairs and  $S$-Camina triples have already been studied in \cite{Burkett2020Vanishing-offProducts} without any nomenclature. An important characterization of the  $S$-Camina triples can also be found in \cite[Proposition 7.3]{Hendrickson2012SupercharacterProducts} which proves that $(G,N,M)$ with $M\leq N$ is an $S$-Camina triple if and only if $S$ is $\triangle$-product over $M$ and $N$.  Similarly, $(G,N)$ is an $S$-Camina pair if and only if $S$ is a $\ast$-product over $N$.
    \par Our next focus is another subgroup, $U(S \mid N)$  with properties similar to $V(S\mid N)$. Inspired by the subgroup $U(G\mid N)$ from \cite[p.~ 798]{burkett2020groups},  we define $U(S \mid N)$, for each $S$-normal subgroup $N$ of $G$,   as the product of all $S$-normal subgroups $H$ such that $V(S \mid H) \leq N$. In connection to this subgroup, we prove a number of results among which the following are the most significant.    
    	\begin{theorem}\label{ugroupp}
		Let $G$ be a non $S$-abelian group. The following hold.
		\begin{enumerate}
			\item  For each $N \triangle_S G$, the subgroup $U(S \mid N)$ is the unique largest
			subgroup $U \leq G$ such that every character in $\mathrm{Irr}(S \mid U)$ vanishes on $G \setminus N$.
			
			\item   For each  $N \triangle_S G$ and every element $g \in G$, we have
			\[
			g \in U(S \mid N) \quad \Longleftrightarrow \quad 
			\text{every character } \chi \in \mathrm{Irr}(S) \text{ with } g \notin \ker(\chi) \text{ vanishes on } G \setminus N.
			\]
		\end{enumerate}
	\end{theorem}
 \begin{theorem}\label{theorem_uprod}
       Let $H, N\triangle_S G$. Then 
\begin{enumerate}
    \item $ H\leq U(S\mid N) $ if and only if $S$ is a $\Delta$-product over $H$ and $N$.
    \item $ N= U(S\mid N) $  if and only if $S$ is a $\ast$-product over $N$.
\end{enumerate} 
 \end{theorem}
        	Finally, we discuss the special case of $VZ(S)$-group. 
			A group $G$ with supercharacter theory $S$ is said to be a \emph{$VZ(S)$-group} if every $\chi\in$ Irr$\left(S\mid [G,S]\right)$ vanishes on $G\setminus Z(S)$. Regarding this, we obtain an equivalence of $VZ(S)$-group in terms of $V(S)$ and $U(S)$.
        \begin{theorem}
	    Let $G$ be a non $S$-abelian group. The following are equivalent:
		\begin{enumerate}
			\item $G$ is a $VZ(S)$-group.
			\item $Z(S) = V(S)$.
			\item $U(S) = [G,S]$.
		\end{enumerate}
	\end{theorem}

	%\par In \cite[proposition 4.2]{bur}, Burkett deduced a number of equivalent conditions for a group to be $S$-nilpotent. Following these, we here establish a few more equivalent results for an $S$-nilpotent group.
	\section{Preliminaries}
    In this section, we recall a few definitions and results from the existing literature. We begin with the theory of supercharacters.
    \begin{definition}\cite{Diaconis2008SupercharactersGroups}
		Let $G$ be a finite group, let $\mathcal{X}$ be a partition of the set Irr$(G)$ of irreducible characters of $G$, and let $\mathcal{Y}$ be a partition of $G$. We call the ordered pair $(\mathcal{X},\mathcal{Y})$ a supercharacter theory $S$ if
		\begin{enumerate}
			\item[$(i)$] $\mathcal{Y}$ contains $\{1\}$, where $1$ denotes the identity element of $G$,
			\item[$(ii)$] $|\mathcal{X}| =|\mathcal{Y}|$,
			\item[$(iii)$] for each $X\in\mathcal{X}$, the character $\sigma_X=\sum_{\chi\in X}\chi(1)\chi$ is constant on each $Y\in\mathcal{Y}$.
		\end{enumerate}
		The characters $\sigma_X$ are called supercharacters and the elements $Y\in\mathcal{Y}$ are called superclasses. %If $(\mathcal{X},\mathcal{Y})$ is a supercharacter theory $S$ on a finite group $G$, then one can easily see that each $Y\in\mathcal{Y}$ must be a union of conjugacy classes of $G$. For a supercharacter theory $S$, we call the superclasses as $S$-classes. 
	\end{definition}
     As in the classical case, the supercharacter values over all superclasses can be arranged in a tabular form, called the supercharacter table. Throughout this paper, let $G$ be a finite group and $S$ be a supercharacter theory of $G$. Also, let $\sigma_1, \sigma_2, \ldots, \sigma_m$ denote the $S$-characters corresponding to the parts  $X_1, X_2, \ldots, X_m$ of $\text{Irr}(G)$ while $K_1, K_2, \ldots, K_m$ represent the $S$-classes of $G$. Then, from \cite[Eq. 2.4]{Fowler2014RamanujanSupercharacters}, we notice that the $S$-characters satisfy a row orthogonality property with respect to the usual Hermitian inner product given by %Recall that a function $f : G \to \mathbb{C}$ is called a \emph{class function} if $f$ is constant on each conjugacy class of $G$. The space of complex-valued class functions on $G$ is endowed with the natural inner product

\begin{align}\label{roworth}
     \langle \sigma_i, \sigma_j \rangle 
    = \frac{1}{|G|} \sum_{r=1}^m |K_r| \, \sigma_i(K_r)\overline{\sigma_j(K_r)}=\delta_{i,j}\,\|X_i\|^2,
\end{align}
where  $ \|X_i\|^2 = \displaystyle\sum_{\chi \in X_i} \chi(1)^2$ and $\delta_{i,j}$ is the standard Kronecker delta.   There exists an extremely useful analog for the column orthogonality of a supercharacter table as well. 
                    \begin{theorem}\label{corth}\cite[Theorem 3.3]{Burkett2020AnTheory} Let $G$ be a group with supercharacter theory $S$. Then for any $g,h\in G$, we have
                    \begin{align*}
                        \frac{\abs{G}}{\abs{Cl_S(g)}} = \sum_{\chi \in Irr(S)} \frac{\chi(g)\overline{\chi(h}) }{\chi(1)}
                    \end{align*}
                    if $h\in Cl_S(g)$, and is 0 otherwise.                        
                    \end{theorem}
    \par In classical group theory, it is common to use normal subgroups and quotient groups to deduce properties about the parent group. Using exactly this idea, we can achieve important results in case of supercharacter theory as well. In this context, we recall that a subgroup $N\leq G$ is called $S$-normal, denoted by $N\triangle_S G$, if $N$ is the union of some $S$-classes of $G$. For an $S$-normal subgroup $N$ of $G$, Hendrickson \cite{Hendrickson2012SupercharacterProducts} established that there is an induced supercharacter theory $S_N$ on $N$, called \textit{restriction} and $S^{G/N}$ on $G/N$, called \textit{deflation} for which $$\text{Cl}(S_N)=\left\{K\in \text{Cl}(S):K\subseteq N\right\}$$ and  $$\text{Cl}(S^{G/N})=\left\{\pi(K):K\in \text{Cl}(S)\right\}.$$ 
    Here and throughout this paper, $\pi:G\rightarrow G/N$ is the usual canonical mapping. If $H$ and $N$ are two $S$-normal subgroups of $G$ such that $H\leq N$, then the induced supercharacter theory on $N/H$, denoted by $S_{N/H}$, is given by $S_{N/H}=(S^{G/H})_{N/H}.$ 
    \par The induced theories described above allow us to take supercharacter theories of $S$-normal subgroups and respective quotient groups to form supercharacter theories of $G$, provided certain suitable conditions are satisfied. One such type of construction of supercharacter theory is the  \emph{$\ast$-product}. The superclasses of $S_N \ast S_{G/N}$ are either $S_N$-classes or preimages of $S_{G/N}$-classes in $G$. It turns out that every 
$S_N \ast S_{G/N}$-class is a union of $S$-classes, and 
that $S$ coincides with $S_N \ast S_{G/N}$ if and only if every 
$S$-class lying outside of $N$ is a union of full $N$-cosets. For another $S$-normal subgroup $M$ such that $M\leq N$, the  \emph{$\Delta$-product} over $M$ and $N$ is a generalisation of $*$-product. For details, see \cite{Hendrickson2012SupercharacterProducts}.
\par The supercharacters of the deflated theory $S^{G/[G,S]}$ can naturally be considered as the set of linear $S$-characters. Thus, this set of characters, denoted by $\text{Irr}(S/[G,S])$, coincides with $\text{Irr}(G/[G,S])$. Burkett has defined a number of supercharacter analogs in \cite{Burkett2020AnTheory, Burkett2018SubnormalityTheory}.
     As with normal groups, Burkett showed that the preservation of quotient structure applies to $S$-normal subgroups as well.
	\begin{lemma}\cite[Lemma 2.1]{Burkett2020AnTheory}\label{lemma1}
		Let $N$ be $S$-normal in $G$, and let $M\leq G$ contain $N$. Then $M$ is $S$-normal if and only if $M/N$ is $S^{G/N}$-normal. 
	\end{lemma}
	In \cite{Burkett2020AnTheory}, Burkett further defined a supercharacter theory analog for the center of a group. For a group $G$ with a supercharacter $S$, Burkett \cite{Burkett2020AnTheory} defined the supercharacter analog $Z(S)$ of center as
	$$Z(S)= \left\{g\in G: \abs{\text{Cl}_S(g)}=1\right\}.$$
    A group $G$ with supercharacter theory $S$ is called $S$-abelian if $Z(S)=G$, otherwise non $S$-abelian. For $H\leq G$, Burkett further defined the supercharacter analog $[H:S]$ of the commutator subgroup as $$[H:S]:=\langle g^{-1}k: g \in H \text{ and } k\in\text{Cl}_S(g)\rangle.$$ 
    %Following this, we also define
    %$$[g,S]=\langle g^{-1}k:  k\in\text{Cl}_S(g)\rangle.$$
    %\begin{conjecture}\label{conj}
	  % For any $g\in G$, we have $[g,S]\triangle_S G$.
	%\end{conjecture}
    %We have the following result on $[G,S]$.
    \begin{lemma}\label{comm}\cite[Corollary 4.2.4]{Burkett2018SubnormalityTheory}
        Let $S$ be a supercharacter theory of $G$. Then $ [G, S]$  can be computed as
\[
[ G, S ] = \bigcap_{\substack{\varphi \in Irr(S/[G,S]) }} \ker(\varphi).
\]

    \end{lemma}
	With these notations, an \emph{$S$-central series} is defined as a series 
	$$G=N_1\geq N_2\geq\cdots\geq N_{r+1}=1$$ of $S$-normal subgroups  if  $N_i/N_{i+1}\leq Z(S^{G/N_{i+1}})$ for all $i=1,2,\ldots, r$.  $G$ is said to be  \emph{$S$-nilpotent} if it has an $S$-central series. The \emph{lower $S$-central series} is a series of subgroups defined inductively as   $\gamma_i(S)=[\gamma_{i-1}(S), S]$ for $i\geq 2$ and $\gamma_1(S)=G$. In the event that the supercharacter theory is clear, we simply write $\gamma_i$.  Similarly, the \emph{upper $S$-central series} is deifned as $\zeta_0(S)$=1,  $\zeta_i(S)/\zeta_{i-1}(S)=Z(S^{G/\zeta_{i-1}(S)})$, for all $i\geq 1$.

		%\item V($S|N$): $<g\in F:$ There exists $\chi\in Irr($S|M$)$ $>$

		%\item $G^{(i)}=[G^{(i-1)}, S^{G^{(i-1)}}]$ where $G^{(1)}=G$
 We recall that $V(S)$ is the smallest subgroup $V \leq G$ such that every character in  $\text{Irr}(S\mid [G,S])$ vanishes on $G \setminus V$.  For future reference, we state some results from \cite{Burkett2020Vanishing-offProducts}.
 \begin{lemma}\label{nvsn}\cite[Lemma 5.2]{Burkett2020Vanishing-offProducts}
Let $S$ be a supercharacter theory of a group $G$, and let $N$ be $S$-normal. Then $N \leq V(S \mid N).$
\end{lemma}
    \begin{theorem}\cite[Proposition 4.2]{Burkett2020Vanishing-offProducts}
Let $S$ be a supercharacter theory of $G$, and let $N$ be $S$-normal. Then 
\[
V(S \mid N) \lhd_{S} G 
\quad \text{and} \quad 
V(S \mid N) = \prod_{\chi \in Irr(S \mid N)} V(\chi),
\]
where 
\[
V(\chi) = \{\, g \in G \mid \chi(g) \neq 0 \,\}.
\]
\end{theorem}
 Finally, we have the following theorem by Burkett and Lewis that serves as a characterization of $S$-Camina triple in terms of $V(S\mid N)$. 	
    \begin{theorem}\cite[Theorem 5.3]{Burkett2020Vanishing-offProducts}\label{burvanish}
        Let $G$ be group with supercharacter theory $S$ such that $M\leq N$. Then the following are equivalent 
	\begin{enumerate}
		\item $S$ is $\triangle$-product over $M$ and $N$; 
		\item For all $g\in G\setminus N$, $\abs{cl_S(g)}=\abs{cl_{S^{G/M}}(gM)}\abs{M}$;  
		\item For each $g\in G\setminus N$ and $m\in M$, there exists $k\in cl_S(g)$ such that $g^{-1}k=m$;
        \item $V(S|M)\leq N$;
        \item For all, $g\in G\setminus N$, $\chi(g)=0$ for any $\chi\in Irr(S|M)$.
	\end{enumerate}
    \end{theorem}
    As a particular case, we have the following.
    \begin{theorem}\cite[Theorem 6.1]{Burkett2020Vanishing-offProducts}\label{burvanishh}
Let $N \triangle_S G$. The following are equivalent:
\begin{enumerate}
    \item $S$ is a $\ast$-product over $N$;
    \item for each $g \in G \setminus N$ and $\chi \in Irr(S \mid N)$, we have $\chi(g) = 0$; and, for each $g \in N$, there exists $\chi \in Irr(S \mid N)$ such that $\chi(g) \neq 0$;
    \item $V(S \mid N) = N$.
\end{enumerate}
\end{theorem}    
                    \section{$S$-GCP}
                     %Motivated by the work of Burkett and Lewis in \cite{Burkett2020Vanishing-offProducts}, we define a few more supercharacter analogs. 	An element $g\in G$ is said to be an \emph{$S$-Camina element} if $\chi(g)=0$, for all $\chi\in$ Irr($S|[G,S]$).	For an $S$-normal subgroup, $N$ of $G$, we call $(G,N)$ to be a \emph{generalised $S$-Camina pair}, abbreviated as $S$-GCP if every element in $G\setminus N$ is an $S$-Camina element.
   This section is dedicated to the study of $S$-GCP. We begin our investigation with the generalization of  \cite[Lemma 2.1]{lewis2009vanishing} that characterizes $S$-Camina elements. %Generalizing, we begin with a characterization of the S-Camina element.
	\begin{lemma}\label{celt}
		Let $G$ be a group with supercharacter theory $S$. Then the following are equivalent.
		\begin{enumerate}
			\item  $g$ is an $S$- Camina element.
			\item $Cl_S(g)=g[G,S]$
			\item $\abs{Cl_S(g)}=\abs{[G,S]}$
			\item For any $z\in [G,S]$, there exists a $y\in Cl_S(g)$ such that $g^{-1}y=z$
		\end{enumerate}
	\end{lemma}
	\begin{proof}
     Let us assume (1). Then, by Theorem \ref{corth}, we have
	\begin{align}\label{eqn_cls}
		\frac{\abs{G}}{\abs{Cl_S(g)}} =& \sum_{\chi \in Irr(S/[G,S])} \frac{|\chi(g)|^2 }{\chi(1)}\notag
		+ \sum_{\chi \in Irr(S \mid [G,S])} \frac{|\chi(g)|^2 }{\chi(1)}\\\notag
        =& \sum_{\chi \in Irr(S/[G,S])} \frac{|\chi(g)|^2 }{\chi(1)} \\\notag
        =& \sum_{\chi \in Irr(G/[G,S])}\frac{|\Tilde{\chi}(g[G,S])|^2 }{\Tilde{\chi}([G,S])} \\
        =&\frac{\abs{G/[G,S]}}{\abs{Cl_{S^{[G/G,S]}}(g[G,S])}}
	\end{align}
    Here, we employed Theorem \ref{corth} for the group $G/[G,S]$ to achieve the last equality. Also, \cite[Proposition 3.11]{Burkett2020AnTheory}) yields $\abs{Cl_{S^{G/[G,S]}}\left(g[G,S]\right)}=1$, thereby proving (3).  %as  $G/[G,S]$ is $S_{G/[G,S]}$-abelian (follows from 
	 Conversely assuming (3), from Theorem \ref{corth}, we obtain $$\displaystyle\sum_{\chi\in Irr(S|[G,S])}\frac{\abs{\chi(g)}^2}{\chi(1)}=0.$$ Since each term in this sum is non-negative, we obtain $\chi(g)=0$ for all $\chi\in$ Irr($S|[G,S]$). This proves (1).
\par Since $G/[G,S]$ is $S^{G/[G,S]}$-abelian, we have $\pi(Cl_S(g))=Cl_{S^{G/[G,S]}}(g[G,S])=\left\{g[G,S]\right\}$. Therefore, we achieve $\pi(x)=x[G,S]=g[G,S]$, for all $x\in Cl_S(g)$. Thus, we must have $g^{-1}x=z$ for some $z\in[G,S]$. This proves $Cl_S(g)\subseteq g[G,S].$ Equivalence of (2) and (3) directly follows from here. 
	\par Finally, assuming (4), we get for any $gz\in g[G,S]$, a $y\in Cl_S(g)$ such that $gz=y\in Cl_S(g)$. Combining this with $Cl_S(g)\subseteq g[G,S]$ proves (2). (2) implies (4) is straightforward. 
	
	\end{proof}
       % Following \cite{lewis2009generalizing}, we define the subgroup vanishing off $Z(S)$.
		
                    %Clearly, $G$ is $VZ(S)$ if and only if $Z(S)=V(S)$. 
  
%Next, following \cite{lewis2009vanishing} we define the supercharacter theory version of Camina elements, Camina pairs, Camina triples and generalised Camina pairs.

   \textbf{Proof of Theorem \ref{corgcp}:}  The proof follows directly from Lemma \ref{celt} and Theorem \ref{burvanish}. % which is also a generalisaiton of \cite[corollary 2.3]{lewis2009vanishing}. 
   	\hfill $\Box$\\\\

The following properties are significant in the study of $V(S)$.
    \begin{lemma}\label{cp}
	%(\textbf{Generalisation of  \cite[Lemma 2.4]{lewis2009vanishing}})
    Let $(G,N)$ be an $S$-GCP. Then the following are true:
		\begin{enumerate}
			\item $[G,S]\leq N$.
			\item For $N\leq M < G$, $(G,M)$ is a $S$-GCP.
			\item If $K$ and $N$ are $S$ normal subgroups such that $K\leq N$, then $(\frac{G}{K}, \frac{N}{K})$ is $S^{G/K}$-GCP.
			\item If $G$ is not $S$-abelian, then $Z(S)\leq N$. %\textbf{Here non S-abelian is needed.}
		\end{enumerate}
	\end{lemma}
	\begin{proof}  For any group $G$ with supercharacter theory $S$, Lemma \ref{nvsn} guarantees $[G,S]\leq V(S)$. Also, definition of $V(S)$ confirms $V(S)\leq N$ and hence (1) follows. Now, we note that  $N\leq M < G$ implies $G\setminus M\leq G\setminus N$, which proves (2).
			\par  For (3), we first note that Lemma \ref{lemma1} implies
            $$\frac{N}{K}\triangle_{S^{\frac{G}{K}}}\frac{G}{K}$$
             So, let $Kg\in \frac{G}{K}\setminus\frac{N}{K}$. Then, for any $\chi\in$ Irr($S|[G,S]$), we have $\chi(g)=0,$ as  $g\notin N$.
			Now, since supercharacters of $G/K$ is identified with supercharacters of $G$ containing $K$ in their kernels. Then,
			for any $\chi\in$ Irr($S|[G,S]$) $\cap$ Irr($S/K$), we have from 
			\begin{align*}
				\Tilde{\chi}(Kg)=\chi(g)=0.
			\end{align*}
			This proves (3).
			\par Finally, for (4) let $g\in Z(S)\setminus N$. Then, 
				$\chi(g)=0$ for all $\chi\in$ Irr($S|[G,S]$). By equation \eqref{eqn_cls}, we have
				\begin{align*}	
		\frac{\abs{G}}{\abs{Cl_S(g)}} =\frac{\abs{G/[G,S]}}{\abs{Cl_{S^{G/[G,S]}}(g[G,S])}}
	\end{align*}
				Since, $g\in Z(S)$, we have $\abs{Cl_S(g)}=1$. Thus, $\abs{[G,S]}=1$,  which is a contradiction as $G$ is not $S$-abelian.
	\end{proof}
   We end this section with the case of $S$-abelian groups.
  \begin{lemma}
      $G$ is $S$-abelian if and only if $(G,1)$ is an $S$-GCP. 
  \end{lemma}
  \begin{proof} If $(G,1)$ is an $S$-GCP, then by Lemma \ref{cp}, we have $[G,S]=1$, which implies  $G$ is $S$-abelian. Conversely, for an $S$-abelian group, we have $\abs{Cl_S(g)}=\abs{[G,S]}=1$ for all $g\neq 1$. By corollary \ref{corgcp}, this completes the proof.
  \end{proof}
    \section{The subgroup $V(S)$}
     Our goal in this section is the study of $V(S)$. We begin with a list of basic properties.
\begin{lemma}\label{vs}%(\textbf{Generalisation of \cite[Lemma 3.2]{lew}})\textbf{Need not restate as a lemma as the proofs are obvious}
		Let $S$ be a supercharacter theory of $G$. Then the following are true
		\begin{enumerate}
			\item $(G, V(S))$ is a $S$-GCP.
			\item $[G,S]\leq V(S)$, 
			\item If $Z(S)\neq G$, then $Z(S)\leq V(S)$.
			%\item Every element for $Z-V(S)$ is an $S$- Camina element for $G$.
			\item If $(G,N)$ is $S$-GCP, then, $V(S)\leq N$.
			\item $V(S)=\bigcap_{N\in \mathcal{N}(S)}N$, where $\mathcal{N}(S)=\left\{N\leq G: (G,N) \text{ is an }  S\text{-GCP}\right\}$.
			%\item $V(S)=\prod_{\chi\in nl(S)}V(\chi)$.
		\end{enumerate}
	\end{lemma}
	\begin{proof}
		\begin{enumerate}
			\item follows from definition.
			\item follows from Lemma \ref{nvsn}.
			\item As $(G, V(S))$ is a $S$-GCP,  the result immediately follows from Lemma \ref{cp}(4). 
			\item holds because $V(S)$ is the smallest subgroup of $G$ outside of which all $\chi\in Irr(S|[G,S])$ vanish, we have $V(S)\leq N$, for all $N\in \mathcal{N}$
			\item  Since, $V(S)$ is the smallest subgroup of $G$ outside of which all $\chi\in Irr(S|[G,S])$ vanish, we have $V(S)\leq N$, for all $N\in \mathcal{N}$.
	Also, since $(G, V(S))$ is an $S$-GCP, we have $\cap_{N\in \mathcal{N}(S)}N\subseteq V(S)$.
			%\item follows from \cite[Proposition 4.2]{Burkett2020Vanishing-offProducts}.
		\end{enumerate}
	\end{proof}
   Our next result relates the upper $S$-central series with $V(S)$.
\begin{theorem}\label{zeta}%(\textbf{Generalisation of \cite[Lemma 3.6]{lewis2009vanishing}})
	If $[G,S]\nleq \zeta_m(S)$ for some $m\geq1$, then $\zeta_{m+1}(S)\leq V(S)$  
\end{theorem}
	\begin{proof}
		For $i=1$, we have $[G,S]\nleq {1} $ implies $\zeta_2(S)=Z(S)\leq V(S)$ due to Lemma \ref{vs}.
		Now, let the statement be true for $i=m-2$.
		Let $[G,S]\nleq \zeta_{m-1}(S)$. Then,  $[G,S]\nleq \zeta_{m-2}(S)$.
	So, by induction hypothesis, we obtain $\zeta_{m-1}(S)\leq V(S) $.
	Since, $(G,V(S))$ is $S$-GCP, by Lemma \ref{cp} we have $\left(\frac{G}{\zeta_{m-1}},\frac{V(S)}{\zeta_{m-1}}\right)$ is $S^{G/\zeta_{m-1}}$-GCP. Therefore, applying Lemma \ref{cp}(4), we achieve
	 \begin{align*}
	 	\frac{\zeta_m(S)}{\zeta_{m-1}(S)}=Z(S^{G/\zeta_{m-1}})\leq \frac{V(S)}{\zeta_{m-1}}.
	 \end{align*}
	Using Lemma \ref{lemma1}, we conclude the proof.
	\end{proof}
   Now, for an $S$-nilpotent group, if we define $S$-\emph{nilpotence class} of $G$ to be the least positive integer $c$ satisfying $\zeta_{c}(S) = G,$ or equivalently $\gamma_{c+1}(S) = 1,$ then the following result is almost immediate.
    \begin{corollary}
Let $G$ be $S$-nilpotent with $S$-nilpotence class $c$. Then $\zeta_{c-1} \leq V(S)$.
\end{corollary}
\begin{proof}
Since, $c$ is the $S$-nilpotent class, we have $[G,S] \nleq \zeta_{c-2}$. We complete the proof using Theorem \ref{zeta}.
\end{proof}
For a group $G$ that is not $S$-nilpotent,  we must have for some positive integer $n$,  $\zeta_n(S) = \zeta_m(S) \text{ for all } m \geq n.$ In that case, we define $\zeta_\infty(S) = \zeta_n(S)$ to be the $S$-\emph{hypercenter} of $G$. We show that if $G$ is not $S$-nilpotent, then $V(S)$ must contain $\zeta_\infty(S)$.

\begin{corollary}
Suppose $G$ is not $S$-nilpotent. Then $\zeta_\infty(S) \leq V(S)$.
\end{corollary}
\begin{proof}
Let $n$ be the integer such that $\zeta_n(S) = \zeta_\infty(S)$. As $G$ is not $S$-nilpotent, we must have $[G,S] \nleq \zeta_n(S)$, and so Theorem \ref{zeta} implies $\zeta_\infty(S)=\zeta_{n+1}(S) \leq V(S)$. 
\end{proof}

    \begin{lemma}\label{vsn} %(\textbf{Generalisation of \cite[Lemma 3.3]{lewis2009vanishing}})
	Let $N$ be an $S$-normal subgroup of $G$ such that $G/N$ is not $S^{G/N}$-abelian. Then, $N\leq V(S)$ and $V(S^{G/N})\leq V(S)/N$. % follows whenever $N\leq V(S)$. 
\end{lemma}
\begin{proof}
	Since $G/N$ is not $S^{G/N}$-abelian, we always find a non-linear supercharacter $\Tilde{\chi}$ of $S^{G/N}$. Also, the supercharacters $\Tilde{\chi}$ of $S^{G/N}$ is identified with supercharacters $\chi$ of $S$ whose kernel contains $N$, we have
	\begin{align*}
		N\leq \text{ker}(\chi)\leq V(\chi)\leq V(S).
	\end{align*} 
Combining this with  Lemma \ref{vs}(1), Lemma \ref{cp}(3) and Lemma \ref{vs}(4),  we  complete the proof.%  ensures that $(G/N, V(S)/N)$ is $S^{G/N}$-GCP.  Also, since $V(S)\leq N$, whenever  $(G,N)$ is $S$-GCP
\end{proof}
	
      %  One has that an $S$-character $\chi$ is linear if and only if $[G, S] \leq \ker(\chi).$
    Using this, we define a series $V_i(S)$ as $V_i(S)=[V_{i-1}(S),S]$ for all $i\geq 2$ and ${V_1}(S)={V}(S),$ which is a supercharacter analog of $V_i(G)$  defined in \cite[p.~1319]{lewis2009vanishing}. When the supercharacter theory is clear, we simply write $V_i$ instead of $V_i(S)$. In the following, we mostly mimic the proof of \cite[Lemma 4.1]{lewis2009vanishing} to establish some fundamental properties of the series $V_i(S)$.  
	\begin{theorem}
	Let $G$ be a group with supercharacter theory $S$. Then the following are true:
	\begin{enumerate}
		\item For all $i\geq 1$, we have
		     \begin{align*}
		     	\gamma_{i+1}\leq V_i\leq \gamma_{i}
		     \end{align*}
		     \item If $V_n< \gamma_{n}$ for some $n$, then 
		     \begin{align*}
		     	V_i< \gamma_{i}
		     \end{align*} for all $i$ with $1\leq i\leq n$.
		     \item If $V_n< \gamma_{n}$, then  $G/V_n$ is $S^{G/V_n}$-nilpotent with $S^{G/V_n}$-nilpotence class $n$ with 
		     \begin{align*}
		     	V_i(S^{G/V_n})=V_i(S)/V_n(S)
		     \end{align*} for $1\leq i \leq n$
	\end{enumerate}
	\end{theorem}
	\begin{proof} %We have from my nilpotency paper.
		%\begin{align*}
		%	 G^{(i)}=\gamma_{i}(S)
		%\end{align*}
		\begin{enumerate}
			\item We use induction on $i$. For $i=1$, we have from Lemma \ref{vs}
			\begin{align*}
				[G,S]=\gamma_2\leq V_1=V(S)\leq \gamma_1=G.
			\end{align*}
			Let the statement  be true for $i=m$. That is, $\gamma_{m+1}\leq V_m\leq \gamma_{m}.$ Therefore, the proof follows from $[\gamma_{m+1},S]\leq [V_m, S]\leq [\gamma_{m},S].$
			\item  Suppose $V_i=\gamma_i$ for some $i$, then by induction it is clear that $V_j=\gamma_j$ for all $j\geq i$. Thus,  $ V_n< \gamma_n$   for some $n$ must imply	$V_i< \gamma_{i}$   for all $1\leq i\leq n$, as $ V_i\leq \gamma_i$  by (1). 
			\item  We know from \cite[Lemma 5.5]{Burkett2020AnTheory} that if $N$ is $S$-normal, then $\gamma_i(S^{G/N}) = \gamma_i(S)N/N$ for each $i \geq 1.$ Using this, we obtain 
            \begin{equation}\label{eqn_vn}
                \gamma_i\left(S^{G/V_n}\right) = \gamma_i(S)V_n/V_n.
            \end{equation}  Note that, for $i=n+1$, we have $  \gamma_{n+1}\left(S^{G/V_n}\right) = V_n$ as $\gamma_{n+1}\leq V_n$. Thus, $G/V_n$ is $S^{G/V_n}$-nilpotent. Also, for any $i\leq n$, equation \eqref{eqn_vn} and $V_n(S)< \gamma_{i}(S)$ ensures that $\gamma_i\left(S^{G/V_n}\right)$ is nontrivial. This shows that $n$ is the $S^{G/V_n}$-nilpotence class.
\par For the second part, we work by induction on $i$. For $i=1$, we obtain  $V(S^{G/V_n})=V(S)/V_n(S)$ from Lemma \ref{vsn}. Suppose now that $i > 1$, and that $V_{i-1}\left(S^{G/V_n}\right) = V_{i-1}(S)/V_n(S).$ Then we have
\begin{align*}
V_i\left(S^{G/V_n)}\right) 
    &= [V_{i-1}\left(S^{G/V_n)}\right),\,S^{G/V_n}] \\[6pt]
    &= [V_{i-1}(S)/V_n(S),\, S^{G/V_n)}] \\[6pt]
    &= [V_{i-1}(S),\, S]V_n(S) / V_n(S) \\[6pt]
    &= V_i(S) / V_n(S),
\end{align*}
which is the desired outcome.            
		\end{enumerate}
	\end{proof}
    Using this along with \cite[Proposition 4.3]{Burkett2020AnTheory},  we achieve a new characterization of $S$-nilpotence groups. 
    \begin{corollary}
        $G$ is $S$-nilpotent if and only if the series $V_i$ terminates at \{1\}.
    \end{corollary}
	%\begin{lemma}
	%	$V_i(S)\leq \gamma_i(S) $ for all $i\in \mathbb{N}$(\textbf{Another characterization for S-nilpotent groups in terms of a vanishing chain})
	%\end{lemma

 For the particular case of $VZ(S)$-group, Theorem \ref{burvanish} immediately yields the following characterization of $VZ(S)$-group in terms of $\triangle$-product of supercharacter theories when $[G,S]\leq Z(S)$.
	\begin{lemma}\label{vzs} Let $S$ be a supercharacter theory of $G$ such that $[G,S]\leq Z(S)$. Then the following are equivalent 
	\begin{enumerate}
		\item $S$ is $\triangle$-product over $[G,S]$ and $Z(S)$. 
		\item For all $g\in G\setminus Z(S)$, $\abs{cl_S(g)}=\abs{cl_{S^{{G}/{[G,S]}}}(g[G,S])}\abs{[G,S]}$.  
		\item For each $g\in G\setminus Z(S)$ and $m\in [G,S]$, there exists $k\in cl_S(g)$ such that $g^{-1}k=m$.
        \item $G$ is a $VZ(S)$ group.
        \item $V(S)\leq Z(S)$.
	\end{enumerate}
	\end{lemma}
    In fact, in Lemma \ref{vzs}, the implication of (4) to (1), (2), and (3), as well as the equivalence between (4) and (5), hold without the assumption that $[G,S]\leq Z(S)$.
	\begin{theorem}\label{zs}
		$G$ is a $VZ(S)$-group if and only if
			  $[G,S]\leq V(S) \leq Z(S).$ %(\textbf{Inspired from \cite[lemma 4]{lewis2009character}})
	\end{theorem}
	\begin{proof}
		  Let $G$ be a $VZ(S)$-group. Then, Lemma \ref{vs}(2) yields $[G,S]\leq V(S)$ and $V(S) \leq Z(S)$ follows from the fact  $V(S)$ is the smallest subgroup $V \leq G$ such that every character in  $\text{Irr}(S|[G,S])$ vanishes on $G \setminus V$. %there cannot be any element in $V(S)\setminus Z(S)$.  % we have $\chi(g)=0$ for all $g\in G\setminus Z(S)$ and for all $\chi\in$ Irr($S|[G,S]$).
	%So, by definition of $V(S)$ we achieve,
	%\begin{align*}
	%	[G,S]\leq V(S)\leq Z(S)
	%\end{align*}
    The converse follows from Lemma \ref{vzs}.
	\end{proof}
  As a consequence of this, we obtain a sufficient condition for $S$-nilpotency of a group.
	\begin{theorem}
	    A $VZ(S)$ group is $S$-nilpotent with $S$-nilpotence class 2.
	\end{theorem}
\begin{proof}
By  Theorem \ref{zs}, we get
  \begin{align*}
 [[G,S], S]\leq [Z(S),S]={1}
  \end{align*}
  Therefore the lower $S$-central series terminates at the trivial subgroup. Hence, G is $S$-nilpotent.
\end{proof}
 It is well known that the character degree reveals a lot about the structure of a finite group. In the setting of supercharacter theory, this has not been explored yet. For our next result, we define the set $\operatorname{scd}(S)=\{\chi(1):\chi\in \text{Irr}(S)\}$.  
\begin{lemma}
    Let $G$ be a $VZ(S)$-group with $m$ $S$-characters of which $n$ are non-linear. Then  $$\operatorname{scd}(S) = \left\{\, 1,\, \|X_1\|\sqrt{|G:Z(S)|}, \|X_2\|\sqrt{|G:Z(S)|}, \ldots, \|X_m\|\sqrt{|G:Z(S)|} \, \right\},$$  where
 $\|X_i\|^2 = \displaystyle\sum_{\chi \in X_i} \chi(1)^2$.

\begin{proof} %Suppose $\chi$ is a non-linear irreducible character of $G$.  
 Let  $K_i,~i=1,2,\ldots m$ be the superclasses of $S$ and $\sigma_i\in Irr(S|[G,S])$. Then, by \cite[Lemma 3.4]{Burkett2020AnTheory} the restriction of $\sigma_i$ to $Z(S)$ satisfies
\[
    Res_{Z(S)}^G(\sigma_i ) = \sigma_i(1)\lambda_i, \qquad \text{for some linear character } \lambda_i \in Irr(Z(S)).
\] Also, by Equation \eqref{roworth}, we have 
\[
     \sum_{r=1}^m |K_r| \, |\sigma_i(K_r)|^2
    = |G|\|X_i\|^2.
\]
Let $Z(S)$ be the union of $K_i, ~i\in I$ where $I\subseteq \{1,2,\ldots, m\}$.
Then, since $(G,Z(S))$ is an $S$-GCP and $|K_i|=1$ for all $i\in I$, we achieve
\begin{align*}
    |G|\|X_i\|^2 
   % &=  \sum_{r=1}^m |K_r| \, |\sigma_i(K_r)|^2\\% \sum_{g \in G} |\sigma_i(g)|^2 \\
    &=  \sum_{r\in I}  \, |\sigma_i(K_r)|^2\\% \sum_{g \in Z(S)} |\sigma_i(g)|^2\\
    &=  \sum_{r\in I}  \, |\sigma_i(1)\lambda_i(K_r)|^2\\% \sum_{g \in Z(S)} |\sigma_i(1)\lambda_i(g)|^2 \\
    &= \sigma_i(1)^2 \, |Z(S)|.
\end{align*}
\end{proof}
\end{lemma}		
	\section{The subgroup $U(S\mid N)$}
    We now study the properties of $U(S \mid N)$, its connection to $V(S\mid N)$ and supercharacter theory products. We recall that
    \begin{align*}
       U(S \mid N) = \prod_{H \in \mathcal{H}} H
		\quad \text{where} \quad
		\mathcal{H} = \{\, H \triangle_S G \mid V(S \mid H) \leq N \,\}. 
    \end{align*}	
    Our investigation begins with establishing the $S$-normality of $U(S\mid N)$.
	\begin{lemma}\label{unormal}
		For each $S$-normal subgroup $N$ of $G$, $U(S \mid N)$ is an $S$-normal subgroup of $G$.
	\end{lemma}
	\begin{proof}
		It follows from \cite[Lemma 3.6]{Burkett2020AnTheory} which states that product of two $S$-normal subgroups is also $S$-normal.
	\end{proof}
	The next result is useful to establish the monotonic behaviour of the subgroup $U(S\mid N)$.
	\begin{lemma}\label{irr}
		Let $M$ and $N$ be $S$-normal subgroups of $G$. Then 
		\[
		M \leq N \quad \text{if and only if} \quad \mathrm{Irr}(S \mid M) \subseteq \mathrm{Irr}(S \mid N).
		\]		
	\end{lemma}
	\begin{proof}
		It is clear that $\mathrm{Irr}(S \mid M) \subseteq \mathrm{Irr}(S \mid N)$ when $M \leq N$.	Conversely, suppose that $\mathrm{Irr}(S \mid M) \subseteq \mathrm{Irr}(S \mid N)$.  
		This implies that $\mathrm{Irr}(S/N) \subseteq \mathrm{Irr}(S/M)$, and so
		\[
		M = \bigcap_{\chi \in \mathrm{Irr}(S/M)} \ker(\chi) 
		\;\leq\; \bigcap_{\chi \in \mathrm{Irr}(S/N)} \ker(\chi) = N.
		\]		
	\end{proof}
    \begin{lemma}\label{uorder}
   Let $S$ be a supercharacter theory of $G$, and let $H$ and $N$ be $S$-normal subgroups.  If $H \leq N$, then $U(S \mid H) \leq U(S \mid N)$.
\end{lemma}
\begin{proof}
    We have 
    \begin{align*}
        V(S|U(S|H))=&V(S|\prod_{K\in\mathcal{H}}K)\\
                =&\prod_{K\in\mathcal{H}}V(S|K) 
    \end{align*}
    where $V(S|K)\leq H\leq N$ for all $K\in\mathcal{H}$. This yields $$V(S|U(S|H))\leq N.$$
    Employing Lemma \ref{ugroup}, we conclude the proof.
\end{proof}
The following result is important as it shows how closely the subgroups $U(S \mid N)$ and $V(S \mid H)$ are linked, where $H$ and $N$ are $S$-normal subgroups of $G$.	
	\begin{lemma}\label{ugroup}
		Let  $H, N\triangle_SG$. Then 
        \begin{align*}
            H \leq U(S \mid N) \quad \Longleftrightarrow \quad V(S \mid H) \leq N.
        \end{align*}
	\end{lemma}
	
	\begin{proof}
		If $V(S \mid H) \leq N$, then it is clear that $H \leq U(S \mid N)$.  
		
		Conversely, suppose that $H \leq U(S \mid N)$. Then by \cite[Lemma 4.3, Lemma 4.4]{Burkett2020Vanishing-offProducts}
		\[
		V(S \mid H) \leq V(S \mid U(S \mid N)) 
		= \prod_{K \in \mathcal{H}} V(S \mid K) \leq N,
		\]
		where 
		\[
		\mathcal{H} = \{\, K \triangle_S G \mid V(S \mid K) \leq N \,\}.
		\]
	\end{proof}
  % The proof of theorem \ref{mproduct} follows as a consequence of  Lemma \ref{ugroup} and theorem \ref{burvanish}. In fact, more can be deduced 
  As a result of Lemma \ref{ugroup}, \ref{nvsn} and Theorem \ref{burvanishh}, we prove Theorem \ref{theorem_uprod}, which we restate   as a corollary to this.
     \begin{corollary}\label{ucorr}
    Let $H$, $N$ be  $S$-normal subgroups of  $G$. Then  $ H\leq U(S\mid N) $ if and only if $S$ is a $\Delta$-product over $H$ and $N$.
\end{corollary}
     \begin{corollary}\label{ucor}
     Let $N$ be an $S$-normal subgroup of $G$. Then $ N= U(S\mid N) $  if and only if $S$ is a $\ast$-product over $N$.
\end{corollary}
Before proving our main results, we recall from \cite[p.~ 863]{Burkett2020AnTheory} that for any $H\subseteq G$,  $H^S$ represents the smallest $S$-normal subgroup of $G$ containing $H$. We call this the $S$-\emph{normal closure} of $H$. Now, we are ready for the proof of Theorem  \ref{ugroupp}.\\
\noindent\textbf{Proof of Theorem \ref{ugroupp}.}
		If every character in $\text{Irr}(S \mid H)$ vanishes on $G \setminus N$, then by Theorem \ref{burvanish}
		$V(S \mid H) \leq N$, so $H \leq U(S \mid N)$ by Lemma \ref{ugroup}. Employing this, Lemma \ref{ugroup} completes the proof of (1).
		
	\par 	To prove (2), first note that for
		$\chi \in Irr(S)$, we have $$g \in \ker(\chi)~~ \text{if and only if }
		\langle g \rangle^S \leq \ker(\chi).$$  Hence, every supercharacter 	$\chi \in Irr(S)$ satisfying $g \notin \ker(\chi)$ vanishes on 
		$G \setminus N$ if and only if every character 
		$\chi \in Irr(S \mid \langle g \rangle^S)$ vanishes on $G \setminus N$. Lemma \ref{ugroup} and Theorem \ref{burvanish} implies that the latter happens if and only if $V(S \mid \langle g \rangle^S) \leq N$,  if and only if $\langle g \rangle^S \leq U(S \mid N)$. Observing that $U(S \mid N)$ is $S$-normal in $G$ by Lemma \ref{unormal}, the desired conclusion follows as $g \in U(S \mid N)$ if and only if$\langle g \rangle^S \leq U(S \mid N)$. 	
		\hfill $\Box$\\  
 \begin{lemma}\label{uorder}
     Let $G$ be a non $S$-abelian group. For each  $N \triangle_S G$, we have
			    $U(S \mid N) \leq N \cap [G,S]$.						
\end{lemma}
\begin{proof}
     The fact that $U(S \mid N) \leq N$ follows from Lemma \ref{ugroup}, as 
		$N \leq V(S \mid N)$ by Lemma \ref{nvsn}.
		By \cite[page 863]{Burkett2020AnTheory}, we know that $\chi\in Irr(S/[G,S])$ are precisely the linear $S$-characters,  and hence such characters do not vanish on any element of $G$. Since, by point (1), for  $\chi\in Irr(S \mid U(S \mid N))$, $\chi$ vanishes on $G\setminus N$, $\chi$ must be non-linear. That is, 
        \begin{align*}
            Irr(S \mid U(S \mid N)) \subseteq Irr(S \mid [G,S]).
        \end{align*}
		Lemma \ref{irr} completes the proof.
\end{proof}

   \begin{theorem}
Let  $N > 1$ be $S$-normal. There exists a non-trivial $S$-normal subgroup $H$ contained in $N$ for which  $S = S_H \,\Delta\, S^{G/N}$ if and only if $U(S \mid N) > 1$.
\end{theorem}
\begin{proof}
Let $U(S \mid N) > 1$.  By Lemma \ref{uorder}, we obtain $ U(S \mid N)\leq N$. If $N = U(S \mid N)$, then Corollary \ref{ucor} implies that $S$ is a $\ast$-product over $N$, which is basically  $\Delta$-product over $N$ and $N$.  On the other hand, if $ U(S \mid N)< N$, then by Corollary \ref{ucorr}, we conclude that $U(S \mid N)$  is the non-trivial $S$-normal subgroup $H$ contained in $N$ over which  $S$ is a $\Delta$-product. The converse follows immediately from Corollary \ref{ucorr}.
\end{proof}

    Now, let us  define the descending chain of subgroups $U^i(S \mid N)$,
where $U^i(S \mid N) = U(S \mid U^{i-1}(S \mid N))$ for  $i \geq 2$ and $U^1(S \mid N)=U(S \mid N)$.
\begin{lemma}
Let $S$ be a supercharacter theory of $G$, and let $1<N < G$ be $S$-normal.
Let $U$ be the last term of $U^i(S \mid N)$. Then $U > 1$ if and only if there exists a non-trivial $S$-normal subgroup $H$ of $G$ contained in $N$ for which $S$ is a $\ast$-product over $H$.
\end{lemma}
\begin{proof}
First, suppose that $U > 1$. Then there is some index $n$ where 
$U^n(S \mid N) = U^{n+1}(S \mid N) = U(S \mid U^n(S \mid N)).$
Then,  Lemma \ref{ugroup}, \ref{nvsn} and Theorem \ref{burvanishh} prove that $S$ is a $\ast$-product over $U$, and clearly $U \leq N$.
\par For the converse part, from Corollary \ref{ucor} we note  that $U(S \mid H) = H$, where $H$ is a non-trivial $S$-normal subgroup  of $G$ contained in $N$ for which $S$ is a $\ast$-product over $H$. Therefore, Lemma \ref{uorder}  yields
$H\leq U^i(S \mid N) $ for all $i$. Since  $U$ is the last term of $U^i(S \mid N)$ and $H$ is a non-trivial subgroup  of $G$ , we have $1 < H \leq U$.
\end{proof}   
	\begin{lemma}
		Let $H$ and $N$ be $S$-normal subgroups of $G$ that satisfy $V(S \mid N) \leq H$.
		Then
		\[
		U(S^{G/N} \mid H/N) \;=\; U(S \mid H)/N.
		\]
	\end{lemma}
	
	\begin{proof}
		%Let $x \mapsto \overline{x}$ denote the canonical surjection $G \to G/N$.  
		Let us define the sets
        \begin{align*}
            	\mathcal{C} =& \{ K \triangle_S G \mid V(S \mid K) \leq H \},\\
                \mathcal{D} = &\{ K/N \triangle_{S^{G/N}} G/N \mid V(S^{G/N} \mid K/N) \leq H/N \},\\
                \mathcal{C}^\ast =&\{ K \triangle_S G \mid N \leq K \text{ and } V(S \mid K) \leq H \}.
        \end{align*}
		We claim that  $\mathcal{D}$ can be identified with $\mathcal{C^\ast}$. %$\mathcal{D} = \mathcal{C}$. However, we first show that , where
		Let $K \triangle_S G$ satisfy $N \leq K$. Then, %\cite[Lemma 4.3]{Bur} and \cite[Lemma 5.2]{Bur} yield $N \leq V(S \mid K)$ and so by 
		\cite[Lemma 4.5]{Burkett2020Vanishing-offProducts} yields
        \begin{align*}
            V(S^{G/N} \mid K/N)V(S \mid N)/N = V(S \mid K)/N.
        \end{align*}
		Using this and the given condition that $V(S \mid N) \leq H$, it follows that $K/N \in \mathcal{D}$ if and only if $K \in \mathcal{C}^\ast$, as claimed.  In particular, this gives
		\begin{align}\label{quo}
			U(S^{G/N} \mid H/N)= \prod_{K/N \in \mathcal{D}} K/N= \prod_{K \in \mathcal{C}^\ast} K/N=\left(\prod_{K\in \mathcal{C}^\ast}K\right)/N.
		\end{align}
				
		Next, we observe that since $V(S \mid N) \leq H$, we have $K \in \mathcal{C}$ if and only if $KN \in \mathcal{C}^\ast$ by \cite[Lemma 4.4]{Burkett2020Vanishing-offProducts}.  
		Hence
		\begin{align}\label{quoo}
			U(S \mid H) = \prod_{K \in \mathcal{C}} K
			= \prod_{K \in \mathcal{C}} KN
			= \prod_{K \in \mathcal{C}^\ast} K.
		\end{align}
		
		The result now follows from \eqref{quo} and \eqref{quoo}.
	\end{proof}	
	\begin{lemma}
		Let $N \lhd_S G$. Then 
		\[
		U(S \mid N) \;=\; \bigcap_{\chi \in \mathcal{U}} \ker(\chi),
		\]
		where 
		\[
		\mathcal{U} \;=\; \{\chi \in Irr(S) \mid V(\chi) \nleq N \}.
		\]
	\end{lemma}
	
	\begin{proof}
	For convenience, let $U = U(S \mid N)$ and let 
		\[
		W = \bigcap_{\chi \in \mathcal{U}} \ker(\chi).
		\]
		We note that $V(S \mid U) \leq N$, and so by Theorem \ref{burvanish}, we achieve $V(\chi) \leq N$ for all characters 		$\chi \in Irr(S \mid U)$. This implies that $U \leq \ker(\chi)$ for every 	character $\chi \in \mathcal{U}$. Hence $U \leq W$.  
		\par Conversely, consider a character $\chi \in Irr(S \mid W)$. Since 
		$W \nleq \ker(\chi)$, it follows that $\chi\notin \mathcal{U}$, that is, $V(\chi) \leq N$. In particular, we showed that
		$\chi$ vanishes on $G \setminus N$ for all $\chi \in Irr(S \mid W)$. By Lemma \ref{ugroupp}, we have $W \leq U$. This completes the proof.
	\end{proof}
	\begin{definition}
		We define
		\[
		U(S) \;=\; U\bigl(S \mid Z(S)\bigr).
		\]
	\end{definition}
	When $G$ is $S$-abelian, $U(S)=G$. So, we assume $Z(S)\neq G$ while studying $U(S)$. 
    \begin{lemma}
		Let $G$ be a non $S$-abelian group. Then
		\[
		U(S) \;\leq\; [G,S] \cap Z(S) \;\leq\; [G,S]Z(S) \;\leq\; V(S).
		\]
	\end{lemma}
	
	\begin{proof}
		We have $U(S) \leq [G,S] \cap Z(S)$  by Lemma\ref{uorder} and $[G,S]Z(S) \leq V(S)$ follows from Lemma \ref{vs}(2, 3).
	\end{proof}	
    We conclude our paper with a characterization of $VZ(S)$-groups in terms of $U(S)$.
	\begin{theorem}
	    Let $G$ be a non $S$-abelian group. The following are equivalent:
		\begin{enumerate}
			\item $G$ is a $VZ(S)$-group.
			\item $Z(S) = V(S)$.
			\item $U(S) = [G,S]$.
		\end{enumerate}
	\end{theorem}
	\begin{proof}
	Employing Lemma \ref{zs} and \ref{vs}(3), we have (1) and (2) are equivalent . Now, let us assume (1). Using Lemma \ref{uorder}, we have $U(S) \leq [G,S]$, while the reverse containment follows from Lemma \ref{ugroup}, since in a $VZ(S)$ group, $V(S)\leq Z(S)$. Now, let (3) be true. Then Lemma \ref{ugroup} implies $Z(S)\geq V(S)$. Therefore, $G$ is a $VZ(S)$-group, completing the proof. 
	\end{proof}
 \section{Acknowledgements}
 The author would like to thank Professor Mark Lewis for his useful suggestions in improving the quality of the paper.

    \bibliography{references.bib}
	\bibliographystyle{ieeetr}

\end{document}